\documentclass[12pt,english]{article}
\usepackage{amsthm}
\usepackage{amsmath}
\usepackage{amssymb}

\makeatletter
\theoremstyle{plain}
\newtheorem{thm}{Theorem}
\theoremstyle{plain}
\newtheorem{conjecture}[thm]{Conjecture}
\theoremstyle{plain}
\newtheorem{lem}[thm]{Lemma}
 \ifx\proof\undefined\
   \newenvironment{proof}[1][\proofname]{\par
     \normalfont\topsep2\p@\@plus2\p@\relax
     \trivlist
     \itemindent\parindent
     \item[\hskip\labelsep
           \scshape
       #1]\ignorespaces
   }{%
     \endtrivlist\@endpefalse
   }
   \providecommand{\proofname}{Proof}
 \fi
\theoremstyle{plain}
\newtheorem{prop}[thm]{Proposition}

\usepackage{amsthm}\usepackage{latexsym}\textheight8.5in\textwidth6.5in\topmargin-0.33in\oddsidemargin0pt\evensidemargin0pt\parindent=0pt\parskip=6pt

\makeatother

\usepackage{babel}

\begin{document}
%  GENERAL MACROS    ++++++++++++++++++++++++++++++++++++++++++++++++++++++

\global\long\def\comment#1{}
 % type \comment[ put your comment here ]
\global\long\def\hs{\enspace}
 \global\long\def\hhs{\thinspace}
 \global\long\def\real{\ifmmode{\rm R} \else${\rm R}$ \fi}

%\newtheorem{theorem}{Theorem}\newtheorem{lemma}[theorem]{Lemma}\newtheorem{corollary}[theorem]{Corollary}\newtheorem{definition}[theorem]{Definition}\newtheorem{claim}[theorem]{Claim}\newtheorem{conjecture}[theorem]{Conjecture}\newtheorem{proposition}[theorem]{Proposition}

%\newtheorem{remark}[theorem]{Remark}

% Put all the macros you use in every paper here.
% I have zillions, so have omitted all but a couple.

%  PAPER SPECIFIC MACROS   +++++++++++++++++++++++++++++++++++++++++++++++++

% All the macros for this paper only.
\global\long\def\edge{\leftrightarrow}
 \global\long\def\noedge{\not\leftrightarrow}
 \global\long\def\twoedge{\Leftrightarrow}
 \global\long\def\to{\rightarrow}
 \global\long\def\Hrl{H_{l+1}^{(r)}}
 \global\long\def\Krl{{\cal K}_{l}^{(r)}}
 \global\long\def\Krl{{\cal K}_{l+1}^{(r)}}
 \global\long\def\cF{{\cal F}}
 \global\long\def\cG{{\cal G}}
 \global\long\def\cH{{\cal H}}
 \global\long\def\cA{{\cal A}}

\global\long\def\e{\varepsilon}
 \global\long\def\bF{ {\cal {\bf F}}}
 \global\long\def\odel{o_{\delta}}
 \def\od 1{o_{\delta}(1)}

\def\oe 1{o_{\varepsilon}(1)} \global\long\def\aF{\alpha_{F}}
 \global\long\def\bF{\beta_{F}}
 \global\long\def\gF{\gamma_{F}}

\baselineskip=18pt
\parskip=8pt
% TITLE INFORMATION  ++++++++++++++++++++++++++++++++++++++++++++++++++++++

\title{\textbf{Toward a Hajnal-Szemer\'edi theorem for hypergraphs}}

\author{H. A. Kierstead \thanks{School of Mathematical and Statistical Sciences, Arizona State University,
Tempe, AZ 85287, USA.
email: kierstead@asu.edu;
research supported in part by NSA grant MDA 904-03-1-0007 and NSF
grant DMS-0901520.} \quad{}and \quad{}Dhruv Mubayi %
\thanks{ Department of Mathematics, Statistics, and Computer Science, University
of Illinois, Chicago, IL 60607. email: mubayi@math.uic.edu; research
supported in part by NSF grants DMS-0653946 and DMS-0969092. \protect \\
 2000 Mathematics Subject Classification:05C15, 05C65, 05C85, 05D40 %\hfil\break\null\hskip .23in Keywords: {\it Hypergraph Tur\'an
%numbers, removal lemma, stability theorems}
%
}}

\date{\today}

\maketitle
%  ABSTRACT   +++++++++++++++++++++++++++++++++++++++++++++++++++++++

\begin{abstract}
 Let $\cH$
be a triple system with maximum degree $d>1$ and let $r>10
^7\sqrt{d}\log^{2}d$.
Then $\cH$ has a proper vertex coloring with $r$ colors such that
any two color classes differ in size by at most one. The bound on
$r$ is sharp in order of magnitude apart from the logarithmic factors.  Moreover, such an $r$-coloring can be found via a randomized algorithm whose expected running time is polynomial in the number of vertices of $\cH$.

 This is the first result  in the direction of generalizing the Hajnal-Szemer\'edi theorem to hypergraphs.
\end{abstract}

\section{Introduction}

One of the basic facts of graph coloring is that every graph $G$ with
maximum degree $d$ has chromatic number at most $d+1$. An equitable $r$-coloring of $G=(V,E)$ is proper coloring with $r$-colors for which each color
class has size $\lfloor |V|/r\rfloor$ or $\lceil |V|/r\rceil$.  A much deeper
result is:
\begin{thm}[Hajnal-Szemer\'edi \cite{Hajnal}]
For every integer $r$ and graph $G$ with maximum degree $d$, if $d<r$ then $G$ has an equitable $r$-coloring.
\end{thm}
The original proof was quite complicated,
and did not yield a polynomial time algorithm for producing the coloring, but
recently Mydlarz and Szemer\'edi \cite{MSz}, and independently Kierstead
and Kostochka \cite{KK-HSz}, found simpler proofs that did yield
polynomial time algorithms. See \cite{KKYbc} for an even simpler
proof. These ideas were combined in \cite{KKMS} to obtain an $O(r|V|^{2})$
time algorithm. Kierstead and Kostochka \cite{KK-O} also strengthened
the Hajnal-Szemer\'edi Theorem by weakening the degree constraint---if
$d(x)+d(y)\leq2r+1$ for every edge $xy$ then $G$ has an equitable
$(r+1)$-coloring.

A $k$-uniform hypergraph ($k$-graph
for short) is a hypergraph whose edges all have size $k$. A proper coloring of a hypergraph is a coloring of its vertices with
no monochromatic edge. Hypergraph coloring has a long history beginning with the seminal results of Erd\H os~\cite{E63, E64} about the minimum number of edges in a $k$-graph that is not 2-colorable.  Apart from giving rise to many of the (still open) major problems in combinatorics, attempts to answer questions in this area have led to fundamental new proof methods, most notably the semi-random or nibble method  and the Lova\'sz Local Lemma~\cite{EL}.
In \cite{EL}, Erd\H{o}s and Lov\'asz obtained, as a corollary to the Local Lemma, that every $k$-graph with maximum degree $d$ has chromatic number at most $3d^{1/(k-1)}$.

In this paper we merge these two important areas of research by studying equitable colorings of hypergraphs.
The situation for hypergraphs is much more complicated. First, we
do not even know sharp bounds for chromatic number in terms of maximum
degree.  Our results deal only with 3-graphs.
An easy consequence of the Local Lemma is that every 3-graph
with maximum degree $d$ has a proper coloring with at most $\sqrt{3ed}=(2.85..)\sqrt{d}$
colors. There appears to be no proof of this that does not use the
Local Lemma. On the other hand, complete 3-graphs show that one needs
at least $\sqrt{d/2}>(0.707..)\sqrt{d}$ colors. It remains an open
problem to obtain the best constant here.

As the above discussion indicates, it is premature to hope for a tight
analogue of the Hajnal-Szemer\'edi theorem for 3-graphs. Nevertheless,
we begin to address the question in this paper. We prove:
\begin{thm}
\label{main} Let
$d\ge2$ and $r$ be integers satisfying $r>10^7\sqrt{d}\log^{2}d$.  Then every $n$ vertex $3$-graph with
maximum degree $d$ has an equitable $r$-coloring. Moreover, such an $r$-coloring can be found via a randomized algorithm whose expected running time is polynomial in $n$.
\end{thm}

It remains an open problem to find a deterministic polynomial time algorithm above.

The bound on $t$ in the theorem is likely not best possible, and we make the following conjecture.
\medskip

\begin{conjecture}
Let
$d$ and $r$ be integers satisfying $r>2.86\sqrt{d}$. Then every $3$-graph with maximum degree
$d$ has an equitable $r$-coloring.
\end{conjecture}
Our approach does not seem to extend to $k$-graphs for $k\ge4$.
Nevertheless, we believe a conjecture similar to the one above holds
for $k\ge4$ as well (see Conjecture \ref{k}).
\bigskip

\textbf{Notation and terminology}.
A $k$-edge or $k$-set, is an edge or set of size $k$.
We associate a hypergraph with its edge set, and refer to $3$-edges
as triads. Fix a hypergraph $\mathcal{H}$ on a vertex set $V$. A
\emph{cover} of $\mathcal{H}$ is a graph $H$ such that every triad
of $\mathcal{H}$ contains an edge of $H$. In this case, every proper
coloring of $H$ is a proper coloring of $\mathcal{H}$. For any $v\in V$,
let $L_{v}=\{xy:vxy\in\cH\}$ be the link graph of $v$, and set $L=\bigcup_{v\in V}L_{v}$. For a graph $G$ let $E_{G}(A,B)$ denote the set of edges of $G$ with one end in $A$ and the other end in $B$. If $G$ is a digraph $\vec E_{G}(A,B)$ denotes the set of diedges with tail in $A$ and head in $B$.

\section{Probabilistic tools}

We will use the Local Lemma \cite{EL} in the (standard) form below:
\begin{thm}
\textbf{(Local Lemma)} \label{ll} Let $\cA_{1},\ldots,\cA_{n}$ be
events in an arbitrary probability space. Suppose that each event
$\cA_{i}$ is mutually independent of a set of all the other events
$\cA_{j}$ but at most $d$, and that $P(\cA_{i})<p$ for all $1\le i\le n$.
If $ep(d+1)<1$, then with positive probability, none of the events
$\cA_{i}$ holds.
\end{thm}
Our second tool is the Kim-Vu inequality \cite{KV}. This is needed
to obtain exponential bounds for sums of not necessarily independent
random variables. Let $\Upsilon=(W,F)$ be a hypergraph of rank $2$,
meaning that each $f\in F$ satisfies $|f|\leq2$. Let $z_{v}$ for
$v\in W$ be independent indicator random variables. Set \[
Z=\sum_{f\in F}\prod_{v\in f}z_{v}\]
 where we allow $f=\emptyset$ in which case the empty product is
1. For $A\subseteq W,|A|\leq2$ let \[
Z_{A}=\sum_{\substack{f\in F\\
f\supseteq A}
}\prod_{v\in f\setminus A}z_{v}.\]
 Let $M_{A}=E(Z_{A})$ and $M_{j}=\max M_{A}$ over all $A$ of size
$j$. Set $\mu=M_{0}=E(Z)$ and put \[
M'=\max\{M_{1},M_{2}\}\quad\quad\hbox{and}\quad\quad M=\max\{\mu,M'\}.\]
 Then for any $\lambda>0$, \begin{equation}
\Pr(|Z-\mu|\geq95\lambda^{2}\sqrt{MM'})\leq20|W|e^{-\lambda}.\label{KimVu}\end{equation}

\section{Proof of Theorem \ref{main}}
In this section we prove the existence of the $r$-coloring guaranteed by Theorem \ref{main}.  In the next section we will prove the algorithmic part of the theorem.

Let $d\ge 2$ and $\cH$ be a 3-graph with vertex set $V$ and maximum degree $d$.  Let $r\ge 10^7 \sqrt d \log^2d$.
Notice that we may assume that $d\ge 10^7$: Otherwise $r>d$, and so $\mathcal H$ has a cover graph  $H$ with maximum degree $d$. By the Hajnal-Szemer\'edi Theorem $H$, and thus $\mathcal H$, has an equitable $r$ coloring.
To further simplify matters, first assume that $r$ divides $n=|V|$ and set $$s=\frac{n}{r}.$$ At the end of this section we will discuss the minor modification that is required in the general case.
Our goal is to color $\mathcal{H}$ with $r$ colors so that every
class has size $s$. This is accomplished in three steps.  Throughout the rest of the proof, set
$$t= \lceil \sqrt{d} \, \,\rceil.$$

{\bf Step 1.} First we
partition $V$ into $t$ sets $X_{1},\dots,X_{t}$ and define a graph $H$ so that  $H[X_{i}\cup X_j]$ is a cover of $\mathcal{H}[X_{i}\cup H_j]$ for each pair $i,j\in [t]$, and $H[X_i]$
  has maximum degree less than $$p=10^{5}\log^{2}d.$$ Actually, we will need the more technical statement of Proposition~\ref{degree}. If $|X_{i}|\geq ps$
and $s\mid|X_{i}|$ then we can use the Hajnal-Szemer\'edi Theorem
to partition $X_{i}$ into independent $s$-sets of $H_{i}$; since $H_{i}$ is a cover of $\mathcal H_{i}$, these are also independent $s$-sets of $\mathcal H_{i}$. So we are left with two
problems: either $X_{i}$ is small or $s\nmid|X_{i}|$ for some $i$.

{\bf Step 2.}
In the second step we move vertices from small $X_{i}$
to large $X_{j}$, preserving our ability to partition all but less than $s$ vertices of $X_{j}$ into independent $s$-sets.  This leaves us with no small parts $X_{i}$.

{\bf Step 3.} Finally
in the last step we shift a small number of vertices from one
class to the next so that all $X_{i}$ are divisible by $s$.

In the next three subsections, we carry out the details.

\subsection{Step 1.}
Set  $a=t$ and note that\[
r=10^7\sqrt d \log^2 d\ge  100tp. \]
Recall that $L=\bigcup_{v\in V} L_v$ where $L_v$ is the link graph of vertex $v$.

{\bf Definition.}
{\em An edge $xy\in L$ is \emph{strong }if
$$|\{z\in V: xyz\in \mathcal{H}\}|\ge a;$$
otherwise it is \emph{weak}.
A triad $vxy\in\mathcal{H}$ is strong
if it contains a \emph{strong} edge of $L$; otherwise it is \emph{weak}.}

Let $G\subseteq L$ be the subgraph of $L$ consisting of strong edges.
Later we will view $G$ as a digraph with edges oriented in both directions.
\begin{lem}
There exists a coloring $f:V\rightarrow[t]$ with classes $X_{i}=\{x:f(x)=i\}$
such that for all vertices $v$ and colors $i$, the following two properties hold:

$(A_{v,i})$ $\mathcal{H}$ has less than $10^{4}\log^{2}d$ weak
triads $vxy$ with $f(x)=i=f(y)$ and

$(D_{v,i})$ $G$ has less than $10\log d$ edges $vx$ with $f(x)=i$.\end{lem}
\begin{proof}
Let $f:V\rightarrow[t]$ be a random coloring obtained as follows:
For each vertex $v$, independently choose $f(v)\in[t]$ so that each
color in $[t]$ has probability $1/t$ of being chosen. We will apply
the Local Lemma to prove that with positive probability $f$ satisfies
$(A_{v,i})$ and $(D_{v,i})$ for all vertices $v$ and colors $i$.
For $v\in V$ and $i\in[t]$, let $\overline{A}_{v,i}$ be the event
that $(A_{v,i})$ fails and $\overline{D}_{v,i}$ be the event that
$(D_{v,i})$ fails.

\textbf{Bound on $P(\overline{A}_{v,i})$:} We use Kim-Vu concentration, where
 $\Upsilon$ is the set of edges $xy$ such that $vxy$ is a weak triad, and $z_{x}$ is the indicator random
variable for vertex $x$ receiving color $i$. Clearly $E(z_{x})=1/t$
for each vertex $x$. Note that the Kim-Vu setting is consistent with
our probability space, as we assign each vertex color $i$ with probability
$1/t$ independently of all other vertices.  Moreover, $Z=\sum_{f\in \Upsilon}\prod_{v\in f}z_{v}$ is the random variable that counts the number of edges $xy \in \Upsilon$ such that $f(x)=f(y)=i$.  Hence $\overline{A}_{v,i}$ is precisely the event that $Z\ge10^{4}\log^{2}d$. For adjacent edges $f,f'\in\Upsilon$, the events $\prod_{v\in f}z_{v}$ and $\prod_{v\in f'}z_{v}$ are not independent, and so we do need the Kim-Vu concentration.

Let $d_{v}$ be the number
of edges in $\Upsilon$.
 For any vertex
subset $A$ of size two, $Z_{A}=1$ by definition of the empty product.  Therefore
$$M_2=1.$$
If $\Upsilon$ has a vertex $x$ of degree at least $a$, then there are at least $a$ edges of $\cal H$ containing both $v$ and $x$. This implies that $vx$ is a strong edge and hence all triads of the form $vxy$ where $xy \in \Upsilon$ are strong.  This contradicts the definition of $\Upsilon$.  We conclude that
 $\Upsilon$ has maximum degree less than $a$ and consequently, $$M_{1}\le \frac{a}{t}=1.$$
Also \[
\mu=M_{0}=E(Z)\le\frac{d_{v}}{t^{2}}\le\frac{d}{t^{2}}\le1.\]
 So \[
M'=\max\{M_{1},M_{2}\}=1,\mbox{ }M=\max\{\mu,M'\}=1\mbox{ and }\sqrt{MM'}=1.\]
Set $\lambda=10\log d$. Since $\mu=1$, we have \[P(\overline{A}_{v,i})=
P(Z\ge10^{4}\log^{2}d)=P(Z\ge100\lambda^{2})\le P(|Z-\mu|\ge95\lambda^{2}).\]

Inequality (\ref{KimVu}) yields \[
P(|Z-\mu|\ge95\lambda^{2})=P(|Z-\mu|\ge95\lambda^{2}\sqrt{MM'}\,)\le20(2d)e^{-\lambda}\le\frac{40d}{d^{10}}<\frac{1}{d^{8}}.\]

\textbf{Bound on $P(\overline{D}_{v,i})$:} The crucial observation
here is that the graph $G$ has maximum degree at most $2d/a$. This
is true because every edge of $G$
is contained in at least $a$ edges of $\cH$. So if a vertex $v$
is incident to more than $2d/a$ edges of $G$, then $d_{\cH}(v)>(2d/a)(a/2)=d$,
contradiction. Consequently,
\begin{align}
P(\overline{D}_{v,i}) &\le P(\mbox{Bin}(2d/a,1/t)>10\log d) \notag \\
&<\binom{2d/a}{10\log d}\left(\frac{1}{t}\right)^{10\log d} \notag \\
&<\left(\frac{2ed}{10at\log d}\right)^{10\log d} \notag \\
&<\left(\frac{1}{e}\right)^{10\log d}<\frac{1}{d^{9}}. \notag
\end{align}
Each $E\in\{\overline{A}_{v,i},\overline{D}_{v,i}\}$ is independent
of any collection of $F\subseteq\{\overline{A}_{w,j},\overline{D}_{x,l}:w,x\in V,j,l\in[t]\}$
as long as no edge containing $w$ or $x$ shares a point with an
edge containing $v$. So we may apply the Local Lemma with dependency
degree at most $5d^{2}$. Since $e(5d^{2})(1/d^{9})<1$, the Local
Lemma implies that there is a vertex partition $X_{1}\cup\ldots\cup X_{t}$
of $V$ that satisfies none of the events $\overline{A}_{v,i},\overline{D}_{v,i}$.
\end{proof}
Fix $f:V\rightarrow[t]$ with color classes $X_{i}:=\{x:f(x)=i\}$
as in the lemma.
%By $(D_{v,i})$ the graph $G[X_{i}]$ has maximum
%degree at most $10\log d$ for each $i$.
Viewing $G$ as a digraph,
let $H\supset G$ be the digraph formed from $G$ by adding the diedges
$(v,x)$ and $(v,y)$ for each weak triad $vxy\in\cH$ with $f(x)=f(y)$.  Occasionally we will view $H$ as a (simple) graph by replacing the diedges $(x,y)$ or $(y,x)$ by the (undirected) edge $xy$.
%Let $H_{i}=H[X_{i}]$.
%We write $N_{i}(v)=N_{H_{i}}(v)\cap X_{i}$,
%$d_{i}(v)=|N_{i}(v)|$, $N_{i}^{+}(v)=N_{H_{i}}^{+}(v)$, $d_{i}^{+}(v)=|N_{i}^{+}(v)|$,
%$N_{i}^{-}(v)=N_{H_{i}}^{-}(v)$, and $d_{i}^{-}(v)=|N_{i}^{-}(v)|$.
%Set\[
%p=10^{5}\log^{2}d.\]
\bigskip

\begin{prop} \label{bicover}
$H[X_{i}\cup X_{j}]$ covers $\mathcal{H}[X_{i}\cup X_{j}]$.\end{prop}
\begin{proof}
Let $xyz$ be a triad in $\mathcal{H}[X_{i}\cup X_{j}]$. If $xyz$
is strong then it contains a strong edge from $G[X_i \cup X_j] \subset H[X_{i}\cup X_{j}]$.
Otherwise $xyz$ is weak, but has two vertices from the same class,
say $x$ and $y$. Then $(z,x),(z,y)\in E(H)$.
\end{proof}
\bigskip

\begin{prop}\label{degree}
All $v\in V$ and $i\in[t]$ satisfy

{\rm (a)} $|\vec E_{H}(v,X_{i})|<p-1$, %$d_{i}^{+}(v)<p-1$,
and

{\rm (b)} if $v\in X_{i}$ then $|E_{H}(v,X_{i})| <p-1$, where we view $H$ as a  graph.%$d_{i}(v)<p-1$.
\end{prop}
\begin{proof}
By property $(D_{v,i})$, we have $|\vec E_{G}(v,X_{i})|<10\log d$. The number of weak triads of $\cH$ that contribute to $\vec E_{H-G}(v,X_{i})$
 is at most
$10^{4}\log^{2}d$ by $(A_{v,i})$,
and each of these triads contributes two out-edges to $\vec E_{H-G}(v,X_{i})$.
Thus \[
|\vec E_{H}(v,X_{i})|\le10\log d+2\cdot10^{4}\log^{2}d<10^{5}\log^{2}d-1=p-1.\]
 Now suppose that $u,v\in X_{i}$, and $(u,v)\in \vec E_{H}(u,X_{i})$. If $(u,v)\in \vec E_{G}(u,X_{i})$
then also $(v,u)\in \vec E_{G}(v,X_{i})$, and so $u$ has already been counted as an out-neighbor of $v$. If
$(u,v)\in \vec E_{H-G}(u,X_{i})$ then there exists a weak triad $uvw\in\mathcal{H}$
with $f(v)=f(w)=i$. Since $f(u)=i$, $(v,u)\in \vec E_{H-G}(v,X_{i})$ and so again
$u$ has been counted as an out-neighbor of $v$.\end{proof}

By Proposition \ref{degree} part (b),
\bigskip

\centerline{the graph $H[X_i]$ has maximum degree less than $p$ for each $i \in [t]$. \quad $(*)$}
\bigskip

By the Hajnal-Szemer\'edi Theorem, for every $p'\geq p$, every $(p's)$-subset of vertices of
 $H[X_{i}]$ has an equitable $p'$-coloring with color classes of
size $s$. To finish the argument we would like to partition each $X_{i}$ into blocks
whose sizes are
 at least $ps$ and is divisible by $s$, but this may not be possible.  So we may need to make some adjustments to
the $X_{i}$'s. This is not too difficult if all the $X_{i}$'s have
size at least $12ps$, but first we must arrange that none of the
$X_{i}$ have size less than $12ps$. This is done in Step 2  by distributing all vertices in small $X_i$ to big $X_j$. Doing so may corrupt the nice
properties of the remaining big $X_{i}$, so we must preserve a reasonably
large uncorrupted segment of each big $X_{i}$.

\subsection{Step 2.} First we get organized. Let $J=\{{i}:|X_{i}|\ge12ps\}$
and $S=V\setminus\bigcup_{i\in J}X_i$. Then $|S|<12pst$. We may assume that $J=[t_{0}]$. For each $i\in [t_{0}]$, partition
$X_{i}$ as $X_{i}=Y_{i}\cup Z_{i}$ so that
\begin{equation} ps\mid|Y_{i}|\textrm{ and }12ps\le|Z_{i}|<13ps.\label{YZ}\end{equation}

Set $Z=\bigcup_{i\in [t_0]}Z_{i}$.
Using the Hajnal-Szemer\'edi Theorem and $(*)$, properly color each $H[Y_{i}]$ (and thus each $\mathcal H[Y_i]$) to
obtain color classes $Y_{i,1},\ldots,Y_{i,j}$, each of size $s$.
Let $t_{1}$ be the number of these classes. Then
$$n=|V|=t_1s+|Z|+|S|\le t_1s+13pst+12pst=t_1s+25pst.$$
Recalling that $n=rs$ and dividing by $s$ yields
\begin{equation}
t_{1}\ge r-25pt\ge 100pt-25pt=75pt.\label{t1}\end{equation}
This also implies that $r-t_1 \le 25pt \le r/4$ and hence $t_1\ge 3r/4$.
 Let us rename the $Y_{i,j}$'s, to obtain a vertex partition $W_{1}\cup\ldots\cup W_{t_{1}}$
of $V-S-Z$.

For $v\in S$, define
$$I(v)=\{i\in[t_{1}]:\vec E_H(v,W_{i})=\emptyset\}.$$
As the $W_{i}$'s were formed by refining the partition given
by the $X_{j}$'s, and $|\vec E_H(v,X_j)|<p$  by Proposition~\ref{degree},  we obtain from \eqref{t1}
\begin{equation} |I(v)|\ge t_1-pt\ge\frac{r}{2}.\label{pr}\end{equation}

Now for each $v\in S$, pick one of the elements $i\in{I}(v)$,
where each $i$ has probability $1/|{I}(v)| $ of being picked,
and assign $v$ the color $i$, calling this assignment $\chi(v)=i$.
In this way each $W_{i}$ is enlarged to a set $W_{i}^{+}$, where
$W_{i}^{+}$ contains all those $v\in S$ for which $\chi(v)=i$.
The sets $W_{i}^{+}$ now partition $V\setminus Z$.
\begin{lem} \label{choice}
There is a choice of $\chi$ so that each $W_{i}^{+}$ is an independent
set of $\cH$. \end{lem}
\begin{proof}
For each triad $e=vxy\in\cH$, with $v\in S$, let ${B}_{e}$
be the event that $e$ becomes monochromatic after these random choices
have been made, i.e., $e\subset W_{i}^{+}$ for some $i\in[t_{1}]$.
By the choice of $\chi(v)\in I(v)$, $\{x,y\}\nsubseteq W_{i}$. If $v,x\in S$ and $y\in W_{i}$, then \[
P(B_{e})=P(\chi(v)=i=\chi(x))=\frac{1}{|{I}(v)|}\frac{1}{|{I}(x)|}\le\frac{4}{r^2}<\frac{1}{100d};\]
 otherwise $v,x,y\in S$, and so again \[
P(B_{e})\le\sum_{i=1}^{t_{1}}P(\chi(v)=\chi(x)=\chi(y)=i)=\frac{t_{1}}{|{I}(v)||{I}(x)||{I}(y)|}<\frac{8r}{r^{3}}=\frac{8}{r{}^{2}}<\frac{1}{100d}.\]
 In both cases $P(B_{e})\le p=1/100d$. The event $B_{e}$
is mutually independent of all other events $B_{f}$ for which $e\cap f=\emptyset$,
so the dependency degree in the Local Lemma is at most $3d-1$. Since
$ep(3d)<1$, the Local Lemma implies that there is a partition $W_{1}^{+}\cup\ldots\cup W_{s}^{+}$
of $V-Z$ that is a proper coloring of $\cH$.
\end{proof}
For each $j\in [t_1]$, $|W_{j}^{+}|\ge|W_{j}|=s$. Partition each
$W_{j}^{+}$ into  $s$-sets and one set $R_{j}$ (possibly
empty) of size less than $s$, so that $R_{j}\subset W_{j}=Y_{i,h}\subseteq Y_i\subseteq X_i$. This
is possible because $|W_{j}|=s$.  Note also that these $s$-sets have been shown to be independent sets of $\cal H$ in Lemma \ref{choice}.

 For each $i\in[t_0]$, set
$$U_{i}=Z_{i}\cup \bigcup_{R_j\subseteq Y_i}R_{j}\quad \text{ and } \quad U=\bigcup_{i\in[t_{0}]}U_{i}.$$
We have partitioned ${\cal H}[V\setminus U]$ into independent $s$-sets. Moreover, we have a partition of $U$ into large subsets $U_{i}\subseteq X_{i}$, each with size at
least $12ps$, as $U_i \supset Z_i$. This completes
Step 2.

\subsection{Step 3.} We have already colored all of ${\cal H}[V\setminus U]$
using classes of size $s$. The following
lemma completes Step 3 and the proof of the theorem.
\begin{lem} \label{9}
There exists a proper coloring of $U$ so that every class has
size $s$. \end{lem}
\begin{proof}
If $s\mid|U_{i}|$ for all $i\in [t_0]$
then, since each $|U_{i}|\ge |Z_i| \ge 12ps$, by $(*)$ we can use the Hajnal-Szemer\'edi Theorem to color $H[U_{i}]$, and thus $\mathcal H[U_{i}]$, so that
every class has size $s$. Otherwise, for all $i\in [t_0-1]$ we plan to shift small subsets
$Q_{i}\subseteq U_{i}$ from $U_{i}$ to $U_{i+1}$ so that

(i) $|Q_i|<s$,  and

(ii) $s\mid|Q_{i-1}\cup U_{i}\setminus Q_{i}|.$

The choice of $Q_{i}$ must be made with some care; since $Q_{i}\cup U_{i+1}$ contains vertices from $X_{i}$ and $X_{i+1}$, the degree bound of Proposition \ref{degree}(b) does not hold for this set. However the out-degree bound (a) does hold for both $X_{i}$ and $X_{i+1}$. So we will be able to choose $Q_{i}\subset U_i\setminus P_i$ and $P_{i+1}\subseteq U_{i+1}$
so that

(iii) $|Q_i\cup P_{i+1}|=4ps$ and

(iv) for all $v\in Q_{i}\cup P_{i+1}$ we have $|E_H(v,Q_{i}\cup P_{i+1})|<4p-1$ (i.e. $\Delta(H[Q_i \cup P_{i+1}]) <4p-1$).

We do this recursively. Initialize by setting $Q_{0}=\emptyset=P_{1}.$
Now suppose we have constructed $Q_{j}\subseteq U_{j}$ and $P_{j+1}\subseteq U_{j+1}$
for all $j<i$ so that (i--iv) hold. Set
$$\rho=|Q_{i-1}\cup U_{i}|\mod s.$$ Let
$P$ be an $(8ps)$-subset of $U_{i}\setminus P_{i}$ and $P'$ be
a $(8ps)$-subset of $U_{i+1}$. By Proposition~\ref{degree}, each vertex $v\in P$ has at most
$p$ neighbors in $P$ and at most $p$ out-neighbors in $P'$. Since in $H$
every vertex of $P'$ has at most $p$ out-neighbors in $P$, $|\vec E_H(P',P)|\leq8p^{2}s$.  Since $\rho \le s \le |P|/2$,
 we can choose a $\rho$-subset $Q_{i}\subseteq P$ so that every
vertex $v\in Q_{i}$ has in-degree from $P'$ satisfying $|\vec E(P',v)|<2p$; so (i) and (ii) are satisfied.
Similarly, we can choose $P_{i+1}\subseteq P'$ such that $|P_{i+1}|=4sp-\rho$
and every vertex $v\in P_{i+1}$ has in-degree from $Q_{i}$ satisfying
$|\vec E(Q_{i},v)|<2p$; so (iii) is satisfied. It follows that the maximum degree of $H[Q_{i}\cup P_{i+1}]$
is less than $4p-1$; so (iv) is satisfied.

By (i,iii,iv) and the Hajnal-Szemer\'edi Theorem, we can color $H[Q_{i}\cup P_{i+1}]$ so that every
class has size $s$. Since $Q_{i}\cup P_{i+1}\subseteq X_{i}\cup X_{i+1}$,
this a proper coloring of $\mathcal{H}[Q_{i}\cup P_{i+1}]$ by Proposition~\ref{bicover}. Similarly, by (ii), we can color each $\mathcal H[Q_{i-1}\cup U_{i}\setminus Q_{i}]$ so that each class has size $s$.

The procedure terminates when we have constructed
$P_{t_0} \subset U_{t_0}$.  By the above, we may equitably color
${\cal H}[Q_{t_0-1} \cup P_{t_0}]$ and so it remains to equitably color ${\cal H}[U_{t_0}\setminus P_{t_0}]$.  Since $s$ divides $n$ and the remaining vertices have been partitioned into $s$-sets, we conclude that $s$ also divides $|U_{t_0}\setminus P_{t_0}|$. Also, $|U_{t_0}\setminus P_{t_0}| \ge 12ps-4ps=8ps$ and by Proposition~\ref{degree} (b), the graph $H[U_{t_0}\setminus P_{t_0}]$ has maximum degree less than $p$. Therefore, we can once again use the Hajnal-Szemer\'edi Theorem to equitably color $H[U_{t_0}\setminus P_{t_0}]$ into $s$-sets.
 \end{proof}

Finally, we consider the case that $n=qr+b,~0< b<r$, and set $s=q+1$. So we need to partition $V$ into $b$ independent $s$-sets and $r-b$ independent $(s-1)$-sets. There is no change in Step 1. In Steps 2 and 3 we begin constructing independent $s$-sets, but after we have constructed $b$ of them, we build blocks with parts divisible by $s-1$. We apply the Hajnal-Szemer\'edi Theorem in exactly the same way, and it even does not matter if the switch comes in the middle of a block.

\section{A randomized algorithm}

In this short section we prove that the $r$-coloring of the previous section can be found via a randomized algorithm whose expected running time is polynomial in $n$.

The $r$-coloring is obtained in the following sequence of steps:

1) use the Local Lemma to produce the partition $X_i \cup \ldots \cup X_t$

2) apply the Hajnal-Szemer\'edi Theorem  to equitably color a large subset $Y_i$ of vertices of the graph $H[X_i]$

3) apply the Local Lemma to insert the vertices of $S=V\setminus \bigcup_{i \in J} X_i$ into the sets $W_i$, and

4) deterministically color $U$ as in Lemma \ref{9} repeatedly using the Hajnal-Szemer\'edi Theorem.

%HK
By the results of  Mydlarz-Szemer\'edi and Kierstead-Kostochka \cite{KK-HSz,MSz,KKMS}, we have deterministic polynomial time procedures for steps 2) and 4) above.  Consequently, it suffices to provide randomized algorithms for Steps 1) and 3), which essentially boils down to derandomizing the Local Lemma in these two instances. This is obtained by applying the recent algorithmic version of the Local Lemma due to Moser and Tardos (Theorem 1.2 in \cite{MT}).

\section{Concluding remarks}

$\bullet$ As mentioned in the introduction, it remains an open question to modify the above algorithm to make it deterministic.
This boils down to obtaining deterministic versions of the Local Lemma applications in the proof.
Moser and Tardos (\cite{MT}, Theorem 1.4) give a deterministic version of their algorithm with two additional technical assumptions. The first is that one should be able to efficiently evaluate the conditional probabilities of bad events given  the values of the random variables on a subset of vertices.  The second is that the maximum degree in the dependency graph for the Local Lemma is bounded by a constant.
Subsequently,  Chandrasekaran, Goyal, and Haeupler~\cite{CGH}
have removed the second assumption above, so it suffices to efficiently compute conditional probabilities given partial information on the values of random variables.

In the case of 3) above, we are able to do this, as it suffices to determine the size of $I(v)$ which amounts to looking at the diedges $(v, w)$ in $\vec H$.
  In the case of 1), we are also able to do this for the events $D_{v,i}$ as this amounts to just checking if vertex $x$ such that $vx \in G$, has color $i$.  However, we are not able to compute this conditional probability for the events $A_{v,i}$ efficiently, due to lack of independence. This is precisely where we needed the Kim-Vu inequalities, and it remains the only bottleneck that prevents us from obtaining a deterministic polynomial time algorithm.

$\bullet$ Surprisingly, our results do not extend to $k$-graphs when $k>3$. Here the number of colors we would expect to use is
$d^{1/(k-1)+o(1)}$. The technical reason why our  proof doesn't seem to work is the following: Assume that $k=4$. In the first application of the Local Lemma, we would need at least $t=d^{1/3+o(1)}$ colors in order to guarantee that the $X_i$ were close to being independent sets.
On the other hand, in the second application of the Local Lemma
(Lemma \ref{choice}) one would have to consider the case that we have edge $e=vwxy \in \cH$ where $v, w \in S$ and $x,y \in W_i$.
Then
$$P(B_e)=P(\chi(v)=i=\chi(w))=\frac{1}{|I(v)||I(w)|}=\frac{1}{t^{2+o(1)}}
=\frac{1}{d^{2/3+o(1)}}\gg \frac{1}{d}.$$  Consequently, the probability is not small enough to apply the Local Lemma.  Nevertheless, we conjecture the following:

\begin{conjecture} \label{k}
For each $k \ge 3$ there exists $c_k>0$ such that for every $d$ and $r \ge c_kd^{1/(k-1)}$ the following holds: Every $k$-graph
with maximum degree $d$ and $n$ vertices has an equitable $r$-coloring.  Such a coloring can be found in deterministic polynomial time in $n$.
\end{conjecture}

$\bullet$ With the recent activity~\cite{KK-HSz, KK-O, KKYbc, KKMS} on results about equitable colorings, it seems appropriate to study the question of obtaining equitable colorings in other contexts where coloring problems were explored. One active area of research is to obtain good upper bounds for the chromatic number of graphs that have local constraints.  In particular, a deep theorem of Johansson~\cite{J} that culminated many years of research is that every triangle-free graph with maximum degree $d$ has chromatic number at most $O(d/\log d)$.  We conjecture the following:

\begin{conjecture} \label{j}
For every $d>1$ there is a constant $c$ such that the every triangle-free graph with maximum degree $d$ has an equitable $r$-coloring, whenever $r>cd/\log d$.
\end{conjecture}

Very recently, Frieze and Mubayi have proved a hypergraph analogue of Johansson's theorem mentioned above.  In particular, they prove in \cite{FM1} and \cite{FM2} that every linear (meaning that every two edges share at most one vertex) $k$-graph with maximum degree $d$ has chromatic number at most $O((d/\log d)^{1/(k-1)})$ and this is sharp in order of magnitude. While the proof of this theorem is much more complicated than Johansson's result, the basic method is similar, so an extension of Conjecture~\ref{j} seems plausible.

\begin{conjecture} \label{fm}
For every $d>1$ there is a constant $c$ such that the every linear $k$-graph with maximum degree $d$ has an equitable $r$-coloring, whenever $r>c(d/\log d)^{1/(k-1)}$.
\end{conjecture}

\end{document}